\documentclass[a4paper,11pt]{amsart}
\usepackage{amstext,amsfonts,amsthm,graphicx,amssymb,amscd,epsfig}
\usepackage{amsmath}
\usepackage[ansinew]{inputenc}    
\usepackage{psfrag}
\usepackage{mathrsfs}

\theoremstyle{definition}
\newtheorem{definition}{Definition}[section]
\newtheorem{ex}[definition]{Example}
\newtheorem{rem}[definition]{Remark}

\theoremstyle{plain}
\newtheorem{prop}[definition]{Proposition}
\newtheorem{lem}[definition]{Lemma}
\newtheorem{coro}[definition]{Corollary}
\newtheorem{teo}[definition]{Theorem}

\newfont{\bbb}{msbm10 scaled\magstephalf}     

\def\R{\mathbb R}

\def\R{\mbox{\bbb R}}

\title{A relation between the curvature ellipse and the curvature parabola}

\author{P. Benedini Riul, R. Oset Sinha}

\date{}

\address{Instituto de Ci\^encias Matem\'aticas e de Computa\c{c}\~ao - USP,
Av. Trabalhador s\~ao-carlense, 400 - Centro,
CEP: 13566-590 - S\~ao Carlos - SP, Brazil}

\email{benedini@usp.br}

\address{Departament de Matemàtiques,
Universitat de Val\`encia, Campus de Burjassot, 46100 Burjassot,
Spain}

\email{raul.oset@uv.es}

\thanks{Work of P. Benedini Riul supported by CAPES - PVE  88887.122685/2016-00}
\thanks{Work of R. Oset Sinha partially supported by DGICYT Grant MTM2015--64013--P}
\subjclass[2000]{Primary 58K05; Secondary 57R45, 53A05} \keywords{immersed surface in 4-space, singular surface in 3-space, curvature ellipse, curvature
parabola, second fundamental form, asymptotic directions}

\begin{document}
\begin{abstract}
At each point in an immersed surface in $\mathbb R^4$ there is a curvature ellipse in the normal plane which codifies all the local second order geometry of the
surface. More recently, at the singular point of a corank 1 singular surface in $\mathbb R^3$, a curvature parabola in the normal plane which codifies all the
local second order geometry has been defined. When projecting a regular surface in $\mathbb R^4$ to $\mathbb R^3$ in a tangent direction corank 1 singularities
appear generically. The projection has a cross-cap singularity unless the direction of projection is asymptotic, where more degenerate singularities can appear.
In this paper we relate the geometry of an immersed surface in $\mathbb R^4$ at a certain point to the geometry of the projection of the
surface to $\mathbb R^3$ at the singular point. In particular we relate the curvature ellipse of the surface to the curvature parabola of its singular projection.
\end{abstract}

\maketitle

\section{Introduction}

In his seminal paper \cite{Little}, J. Little studied the second order geometry of submanifolds immersed in Euclidean spaces, in particular of immersed surfaces in $\mathbb R^4$. He defined the curvature ellipse of a surface $N\subset \mathbb R^4$ at a point $p$, as the curve formed by the curvature vectors of the normal sections of $N$ by the hyperplane $\langle\theta\rangle\oplus N_pN$, where $\theta\in[0,2\pi]$ parametrises the unit circle in $T_pN$. This is a plane curve whose trace is contained in the normal
plane $N_{p}N$ at $p$ and which may degenerate into a segment (radial or not) or even a point. The curvature ellipse contains all
the second order geometrical information of the surface. Isometric invariants of the curvature ellipse,
such as its area, are isometric invariants of the surface. Also, the position
of the point $p$ with respect the curvature ellipse (outside, on, inside) gives
an isometric invariant partition of the surface. This paper has inspired
a lot of research on the subject (see \cite{BruceNogueira,BruceTari,GarciaMochidaFusterRuas,MochidaFusterRuas,MochidaFusterRuas2,BallesterosTari,OsetSinhaTari,RomeroFuster}, amongst others).

Martins and Nuño-Ballesteros in \cite{MartinsBallesteros} define a
curvature parabola for corank 1 surfaces $M\subset\mathbb{R}^{3}$
which contains all the second order geometrical information. This
object is also a plane curve, its trace lies in the normal space of
$M$ and may degenerate into a half-line, a line or a point. In that
paper the authors defined the second fundamental form for a corank 1
singular surface in $\R^3$ and used it to define asymptotic and
binormal directions and the umbilic curvature, which are used to
obtain results regarding the contact of the surface with planes and
spheres.

When projecting orthogonally an immersed surface $N$ in $\mathbb R^4$ to $\mathbb R^3$, the composition of the parametrisation of $N$ with the projection can be seen locally as a map germ $(\mathbb R^2,0)\rightarrow (\mathbb R^3,0)$. If the direction of projection is tangent to the surface, singularities appear. In particular, if the direction is asymptotic, singularities more degenerate than a cross-cap (or Whitney umbrella) appear (\cite{BruceNogueira,mondthesis}).

It is natural to wonder wether there is any relation between the curvature ellipse at a point $p\in N\subset \mathbb R^4$ and the curvature parabola at the projection of the point in the singular projection. In this paper we relate the geometry of the surface in $\mathbb R^4$ to the geometry of the singular projected surface in $\mathbb R^3$. We establish relations between asymptotic and binormal directions and prove the following Theorem relating the curvature loci:

\begin{teo}\label{main}
Let $X:U\rightarrow \R^4$ be the parametrisation of a regular surface $N=X(U)$ and consider $p\in N$. Let ${\bf v}\in T_pN$ and consider $\pi_{\bf
v}:N\rightarrow \R^3$ the orthogonal projection of $N$ to $\R^3$. Let $\Delta_e$ be the curvature ellipse of $N$ at $p$ and let $\Delta_p$ be the curvature
parabola of $\pi_{\bf v}(N)$ at $\pi_{\bf v}(p)$. The following hold:
\begin{enumerate}
\item[i)] $\Delta_e$ is an ellipse with $p$ lying inside it if and only if $\Delta_p$ is a parabola with $\pi_{\bf v}(p)$ lying inside it.
\item[ii)] $\Delta_e$ is an ellipse with $p$ lying outside it or a segment whose line does not contain $p$ if and only if $\Delta_p$ is a parabola with
    $\pi_{\bf v}(p)$ lying outside it (when $\bf v$ is not asymptotic) or a half-line whose line does not contain $\pi_{\bf v}(p)$ (otherwise).
\item[iii)] $\Delta_e$ is an ellipse with $p$ lying on it if and only if $\Delta_p$ is a parabola with $\pi_{\bf v}(p)$ lying on it (when $\bf v$ is not
    asymptotic) or a line which does not contain $\pi_{\bf v}(p)$ (otherwise).
\item[iv)] $\Delta_e$ is an segment whose line contains $p$ or a point different from $p$ if and only if $\Delta_p$ is a line which contains $\pi_{\bf
    v}(p)$, a half-line whose line contains $\pi_{\bf v}(p)$ or a point different from $\pi_{\bf v}(p)$.
\item[v)] $\Delta_e$ is the point $p$ if and only if $\Delta_p$ is the point $\pi_{\bf v}(p)$.
\end{enumerate}
\end{teo}

Using our results we explain and recover results from \cite{BallesterosTari} and \cite{OsetSinhaTari}.

\section{Preliminaries}\label{section-notation}

\subsection{Curvature ellipse}

Given a smooth surface $N\subset\mathbb{R}^{4}$ and $X:U\rightarrow\mathbb{R}^{4}$
a local parametrisation of $N$ with $U\subset\mathbb{R}^{2}$ an open subset, let
$\{e_{1},e_{2},e_{3},e_{4}\}$ be an orthogonal frame of $\mathbb{R}^{4}$ such that at any $u\in U$,
$\{e_{1}(u),e_{2}(u)\}$ is a basis for $T_{p}N$ and $\{e_{3}(u),e_{4}(u)\}$ is a basis for
$N_{p}N$ at $p=X(u)$.
The second fundamental form of $N$ is the vector valued quadratic form
$II_{p}:T_{p}N\rightarrow N_{p}N$ given by
$$II_{p}(\textbf{w})=(l_{1}w_{1}^{2}+2m_{1}w_{1}w_{2}+n_{1}w_{2}^{2})e_{3}+(l_{2}w_{1}^{2}+2m_{2}w_{1}w_{2}+n_{2}w_{2}^{2})e_{4},$$
where $l_{i}=\langle X_{xx},e_{i+2}\rangle,\ m_{i}=\langle X_{xy},e_{i+2}\rangle$
and $n_{i}=\langle X_{yy},e_{i+2}\rangle$ for $i=1,2$ are
called the coefficients of the second fundamental form with respect to the
frame above and $\textbf{w}=w_{1}e_{1}+w_{2}e_{2}\in T_{p}N$. The matrix of the second fundamental
form with respect to the orthonormal frame above is given by
$$
\alpha=\left(
         \begin{array}{ccc}
           l_{1} & m_{1} & n_{1} \\
           l_{2} & m_{2} & n_{2} \\
         \end{array}
       \right).
$$

Consider a point $p\in N$ and the unit circle $S^{1}$ in $T_{p}N$ parametrised by
$\theta\in [0,2\pi]$. The curvature vectors $\eta(\theta)$ of the normal sections of
$N$ by the hyperplane $\langle\theta\rangle\oplus N_{p}N$ form an ellipse in the
normal plane $N_{p}N$, called the \emph{curvature ellipse} of $N$ at $p$, that can also
be seen as the image of the map
$\eta:S^{1}\subset T_{p}N\rightarrow N_{p}N$, where
\begin{equation}\label{ellipse}
\eta(\theta)=\sum_{i=1}^{2}(l_{i}\cos^{2}(\theta)+2m_{i}\cos(\theta)\sin(\theta)+n_{i}\sin^{2}(\theta))e_{i+2}.
\end{equation}
Note that, if we write $\textbf{u}=\cos(\theta)e_{1}+\sin(\theta)e_{2}\in S^{1}$,
$II_{p}(\textbf{u})=\eta(\theta)$.

The classification of the points in the surface is made using the curvature ellipse:

\begin{definition}
A point $p\in N$ is called semiumbilic if the curvature ellipse is a line segment which
does not contain $p$. If the curvature ellipse is a
radial segment, the point $p$ is called an inflection point. An
inflection point is of real type, (resp. imaginary type,
at) if $p$ is an interior point of the
radial segment, (resp. does not belong to it, is one of its end points).
When the curvature ellipse reduces to a point, $p$ is called umbilic. Moreover,
if the point is $p$ itself, then $p$ is said to be a
flat umbilic. A non inflection point $p\in N$ is called elliptic (resp. hyperbolic,
parabolic) when it lies inside (resp. outside, on) the curvature ellipse.
\end{definition}

The \emph{resultant} is a scalar invariant of the surface defined by Little in \cite{Little}, given by
$$\Delta=\frac{1}{4}(4(l_{1}m_{2}-m_{1}n_{2})(m_{1}n_{2}-n_{1}m_{2})-(l_{1}n_{2}-n_{1}l_{2})^{2}).$$
The resultant has the following property: $p$ is a point on the curvature ellipse if and only if
$\Delta(p)=0$. Moreover, the signal of $\Delta$ determines if $p$ lies inside the
ellipse $(\Delta(p)>0)$ or outside $(\Delta(p)<0)$. Therefore, a point $p\in N$ is hyperbolic or elliptic according to whether $\Delta(p)$ is negative
or positive, respectively. If $\Delta(p)$ is equal to zero, the point is parabolic
or an inflexion, according to the rank of $\alpha$: $p$ is parabolic if the rank
is $2$ and an inflection if it is less than $2$.


The curvature ellipse and the resultant are invariant under the action
of the geometric subgroup $SO(2)\times SO(2)$ on $T_{p}N\times N_{p}N$.
However, if the interest is \emph{affine} invariants, that is, properties that
remain invariant under the action of $GL(2,\mathbb{R})\times GL(2,\mathbb{R})$, some differences appear:
the concept of an inflection point is an affine invariant, but semi-umblicity is not.

The curvature ellipse can also be seen as the image of the unit circle in $T_{p}N$ under a map
defined by a pair of quadratic forms $(Q_{1},Q_{2})$. This pair of quadratic forms is the
$2$-jet of the $1$-flat map $F:(\mathbb{R}^{2},0)\rightarrow(\mathbb{R}^{2},0)$ whose graph,
in orthogonal coordinates, is locally the surface $N$. Each point on the surface determines
a pair of quadratics $(Q_{1},Q_{2})=(l_{1}x^{2}+2m_{1}xy+n_{1}y^{2},l_{2}x^{2}+2m_{2}xy+n_{2}y^{2})$ and
the group $\mathcal{G}=GL(2,\mathbb{R})\times GL(2,\mathbb{R})$ acts on these pairs of binary
forms $(Q_{1},Q_{2})$ and provides the
$\mathcal{G}$-orbits listed in Table \ref{tab:quadraticforms}. A point is called an elliptic/hyperbolic/parabolic/inflection point according to the classification of its corresponding orbit as in
Table \ref{tab:quadraticforms}. This definition coincides with the definition given by the relative position to the curvature ellipse (\cite{BruceNogueira}).

\begin{table}[tp]
\caption{The $\mathcal{G}$-classes of pairs of quadratic forms.}
\centering
{\begin{tabular}{ccc}
\hline
$\mathcal{G}$-class & Name\\
\hline
$(x^{2},y^{2})$ & hyperbolic point\cr
$(xy,x^{2}-y^{2})$ & elliptic point\cr
$(x^{2},xy)$ & parabolic point\cr
$(x^{2}\pm y^{2},0)$ & inflection point\cr
$(x^{2},0)$ & degenerate inflection\cr
$(0,0)$ & degenerate inflection\cr
\hline
\end{tabular}
}
\label{tab:quadraticforms}
\end{table}

A tangent direction $\theta$ at $p\in N$ is called an \emph{asymptotic direction}
at $p$ if $\eta(\theta)$ and $\frac{d\eta}{d\theta}(\theta)$ are linear dependent vectors in
$N_{p} N$, where $\eta(\theta)$ is a parametrisation of the curvature ellipse
as in (\ref{ellipse}). A curve on $N$ whose tangent at each point is an asymptotic
direction is called an asymptotic curve.

The following theorem gives a characterization for asymptotic curves for regular
surfaces in $\mathbb{R}^{4}$.

\begin{lem}\label{asymptotic1}
Let $X:U\rightarrow\mathbb{R}^{4}$ be a local parametrisation of a surface $N$
and denote by $l_{1},\ m_{1},\ n_{1},\ l_{2},\ m_{2},\ n_{2}$ the coefficients
of its second fundamental form with respect to any frame
$\{X_{x},X_{y},\textbf{f}_{3},\textbf{f}_{4}\}$ of $T_{p}N\times N_{p}N$
which depends smoothly on $p=X(x,y)$. Then the asymptotic curves of $N$ are
are the solutions curves of the binary differential equation:
\begin{equation}\label{BDE}
(l_{1}m_{2}-l_{2}m_{1})dx^{2}+(l_{1}n_{2}-l_{2}n_{1})dxdy+(m_{1}n_{2}-m_{2}n_{1})dy^{2}=0,
\end{equation}
which can also be written as the following determinant form:
$$
\left|
  \begin{array}{ccc}
    dy^{2} & -dxdy & dx^{2} \\
    l_{1} & m_{1} & n_{1} \\
    l_{2} & m_{2} & n_{2} \\
  \end{array}
\right|=0.
$$
\end{lem}

The discriminant of the differential equation (\ref{BDE}) coincides with the resultant $\Delta$ and so there are $0/1/2/\infty$ asymptotic directions according to wether the point is elliptic/parabolic/hyperbolic (or semiumbilic)/inflection.

Asymptotic directions can also be described via the singularities of projections
to hyperplanes.

\begin{teo}[\cite{mondthesis},\cite{BruceNogueira}]\label{projection}
A tangent direction $v$ at $p$ on $N$ is an asymptotic direction if and only if
the projection in the direction $v$ yields a singularity more degenerate than
a cross-cap.
\end{teo}
Generically the singularities that appear in the projection are those of $\mathcal A_e$-codimension less than
or equal to 3 in Mond's list (See Table \ref{tab:Mondcodimle3}).

\begin{table}[tp]
\caption{Classes of $\mathcal A$-map-germs of $\mathcal A_e$-codimension $\le 3$ (\cite{mond}).}
\centering
{\begin{tabular}{ccc}
\hline
Name & Normal form & ${\mathcal A}_e$-codimension\\
\hline
Immersion & $(x,y,0)$ &0\cr
Crosscap & $(x,y^2,xy)$ &0\cr
$S^{\pm}_{k}$ & $(x, y^2, y^3 \pm x^{k+1}y)$, $k=1,2,3$ & $k$\cr
$B^{\pm}_{k}$ & $(x, y^2, x^2y \pm y^{2k+1})$, $k=2,3$ & $k$ \cr
$C^{\pm}_{3}$ & $(x, y^2, xy^3 \pm x^{3}y)$ & $3$\cr
$H_k$ & $(x,xy+y^{3k-1},y^3)$, $k=2,3$ & $k$\cr
$P_3$ *& $(x,xy+y^3,xy^2+ay^{4}),\,a\ne 0,\frac{1}{2},1,\frac{3}{2}$ & $3$\cr
\hline
\end{tabular}
}

* The codimension of $P_3$ is that of its stratum.
\label{tab:Mondcodimle3}
\end{table}

A usual tool for getting geometrical information of a smooth surface $N$ is
to study their generic contacts with hyperplanes. Such contact is measured
by the singularities of the height function on $N$. Let $X:U\rightarrow\mathbb{R}^{4}$
a local parametrisation of $N$, the family of height functions $H:U\times S^{3}\rightarrow\mathbb{R}$
is given by $H(u,v)=\langle X(u),v\rangle$.
For $v$ fixed, we have the height function $h_{v}$ on $N$ given by $h_{v}(u)=H(u,v)$.
A point $p=X(u)$ is a singular point of $h_{v}$ if and only if $v$ is a
normal vector to $N$ at $p$.
A hyperplane orthogonal to the direction $v$ is an \emph{osculating
hyperplane} of $N$ at $p = X(u)$ if it is tangent to $N$ at $p$ and $h_{v}$ has a
degenerate (i.e., non Morse) singularity at $u$. In such case we call the direction $v$ a
\emph{binormal} direction of $N$ at $p$.

%

\subsection{Corank $1$ surfaces in $\mathbb{R}^{3}$}
Here we present a brief study of surfaces in $\mathbb{R}^{3}$ with corank $1$ singularities.
For more details, see \cite{MartinsBallesteros}.
Given a corank $1$ surface $M\subset\mathbb{R}^{3}$ at $p\in M$, we shall
assume it as the image of a smooth map $g:\tilde{M}\rightarrow\mathbb{R}^{3}$,
with $\tilde{M}$ is a regular surface and $q\in\tilde{M}$ is a corank $1$ point
of $g$ such that $g(q)=p\in M$. Taking $\varphi:U\subset\tilde{M}\rightarrow\mathbb{R}^{2}$
a coordinate system, where $U$ is an opened neighborhood of $q\in\tilde{M}$, $f=g\circ\varphi^{-1}$
is a local parametrisation of $M$ at $p$.

The tangent line to $M$ at $p$ is the set $T_{p}M=\mbox{im}(dg_{q})$, where
$dg_{q}:T_{q}\tilde{M}\rightarrow T_{p}\mathbb{R}^{4}$. The normal plane
$N_{p}M$ is the subspace that satisfies  $T_{p}\mathbb{R}^{4}=T_{p}M\oplus N_{p}M$.
The \emph{first fundamental form} $I:T_{q}\tilde{M}\times T_{q}\tilde{M}\rightarrow\mathbb{R}$
is given by
$$
I(X,Y)=\langle dg_{q}(X),dg_{q}(Y)\rangle,\ \forall\ X,Y\in T_{q}\tilde{M}.
$$
If $f=g\circ\varphi^{-1}$ is a local parametrisation of $M$ at $p$ as before and
$\{\partial_{x},\partial_{y}\}$ is a basis for $T_{q}\tilde{M}$, the coefficients
of the first fundamental form with respect to $\varphi$ are:
$$
\begin{array}{c}
E(q)=I(\partial_{x},\partial_{x})=\langle f_{x},f_{x}\rangle(\varphi(q)),\ F(q)=I(\partial_{x},\partial_{y})=\langle f_{x},f_{y}\rangle(\varphi(q))\\
G(q)=I(\partial_{y},\partial_{y})=\langle f_{y},f_{y}\rangle(\varphi(q)),
\end{array}
$$
and taking $X=a\partial_{x}+b\partial_{y}\in T_{q}\tilde{M}$,
$I(X,X)=a^{2}E(q)+2abF(q)+b^{2}G(q)$.
Let $\perp:T_{p}\mathbb{R}^{3}\rightarrow N_{p}M$, be the orthogonal projection onto
the normal plane. The \emph{second fundamental form} of $M$ at $p$
$II:T_{q}\tilde{M}\times T_{q}\tilde{M}\rightarrow N_{p}M$ is the symmetric bilinear map
such that
$$\begin{array}{c}
II(\partial_{x},\partial_{x})=f_{xx}^{\perp}(\varphi(q)),\ II(\partial_{x},\partial_{y})=f_{xy}^{\perp}(\varphi(q))\ \mbox{and}\
II(\partial_{y},\partial_{y})=f_{yy}^{\perp}(\varphi(q)).
\end{array}
$$

The definition of the second fundamental form does not depend
on the choice of local coordinates on $\tilde{M}$. Futhermore, given
a vector $\nu\in N_{p}M$, we define the \emph{second fundamental form in
the direction $\nu$} of $M$ at $p$: $II_{\nu}:T_{q}\tilde{M}\times T_{q}\tilde{M}\rightarrow\mathbb{R}$
whose expression is $II_{\nu}(X,Y)=\langle II(X,Y),\nu\rangle$, for all
$X,Y\in T_{q}\tilde{M}$. The coefficients of $II_{\nu}$ in coordinates are
$$l_{\nu}(q)=\langle f_{xx}^{\perp},\nu\rangle(q),\ m_{\nu}(q)=\langle f_{xy}^{\perp},\nu\rangle(q)\ \mbox{and}\ n_{\nu}(q)=\langle
f_{yy}^{\perp},\nu\rangle(q).$$
For $X=a\partial_{x}+b\partial_{y}\in T_{q}\tilde{M}$, we have
$II_{v}(X,X)=a^{2}l_{\nu}(q)+2abm_{\nu}(q)+b^{2}n_{\nu}(q)$ and fixing an
orthonormal frame $\{\nu_{1},\nu_{2}\}$ of $N_{p}M$,
$$\begin{array}{cl}\label{2ff}
II(X,X) & =II_{\nu_{1}}(X,X)+II_{\nu_{2}}(X,X)\\
        & =(a^{2}l_{\nu_{1}}+2abm_{\nu_{1}}+b^{2}n_{\nu_{1}})\nu_{1}+(a^{2}l_{\nu_{2}}+2abm_{\nu_{2}}+b^{2}n_{\nu_{2}}),
\end{array}
$$
with the coefficients calculated in $q$. We also can represent the
second fundamental form by the matrix of coefficients
$$
\left(
  \begin{array}{ccc}
    l_{\nu_{1}} & m_{\nu_{1}} & n_{\nu_{1}} \\
    l_{\nu_{2}} & m_{\nu_{2}} & n_{\nu_{2}} \\
  \end{array}
\right).
$$

The \emph{curvature parabola} is the set $\Delta_{p}\subset N_{p}M$ given by
$\eta_{q}(C_{q})$ where $C_{q}\subset T_{q}\tilde{M}$ is the subset of unit
vectors and $\eta_{q}:C_{q}\rightarrow N_{p}M$ is defined by
$\eta_{q}(X)=II(X,X)$.

The curvature parabola is a plane curve that can degenerate into a line,
a half-line or a point. The definition does not depend on the choice of coordinates
for $\tilde{M}$, however it depends on the map $g$ which parametrises $M$.
Since $g$ has corank $1$ at $q\in\tilde{M}$, it is possible to choose a coordinate
system and make rotations in $\mathbb{R}^{3}$ in a way that $f(x,y)=(x,f_{2}(x,y),f_{3}(x,y))$
and $(f_{i})_{x}(\varphi(q))=(f_{i})_{y}(\varphi(q))=0$ for $i=2,3$. Hence,
we obtain $E=1$, $F=G=0$ and $C_{q}=\{X=(\pm1,y)|\ y\in\mathbb{R}\}$. Therefore,
fixing an orthonormal frame $\{\nu_{1},\nu_{2}\}$ of $N_{p}M$ and using (\ref{2ff}),
\begin{equation}\label{parabola}
\eta(y)=(l_{\nu_{1}}+2m_{\nu_{1}}y+n_{\nu_{1}}y^{2})\nu_{1}+(l_{\nu_{2}}+2m_{\nu_{2}}y+n_{\nu_{2}}y^{2})\nu_{2}
\end{equation}
is a parametrisation for $\Delta_{p}$ in $N_{p}M$.

In \cite{mond}, Mond showed that all corank $1$ map germs $f:(\mathbb{R}^{2},0)\rightarrow(\mathbb{R}^{3},0)$
can be partitioned according to its $2$-jets, $j^{2}f(0)$, under the action of
$\mathcal{A}^{2}$, the space of $2$-jets of diffeomorphisms in
source and target. The space of $2$-jets $j^{2}f(0)$ of map germs
$f:(\mathbb{R}^{2},0)\rightarrow(\mathbb{R}^{3},0)$ is denoted by
$J^{2}(2,3)$ and $\Sigma^{1}J^{2}(2,3)$ is the subset of $2$-jets of corank $1$.

\begin{prop}[\cite{mond}]
There exist four $\mathcal{A}^{2}$-orbits in $\Sigma^{1}J^{2}(2,3)$:
$$(x,y^{2},xy),\ (x,y^{2},0),\ (x,xy,0)\ \mbox{and}\ (x,0,0).$$
\end{prop}

The next theorem is a powerful tool to distinguish the four orbits in
the previous proposition just using the curvature parabola. It shows that
the curvature parabola is a complete invariant for the Mond's $\mathcal{A}^{2}$-
classification.

\begin{teo}[\cite{MartinsBallesteros}]\label{conditionsparabola}
Let $M\subset\mathbb{R}^{3}$ be a surface with a singularity of corank $1$ at $p\in M$. We
assume for simplicity that $p$ is the origin of $\mathbb{R}^{3}$ and denote by $j^{2}f(0)$
the $2$-jet of a local parametrisation $f:(\mathbb{R}^{2},0)\rightarrow(\mathbb{R}^{3},0)$
of $M$. Then the following holds:
\begin{itemize}
\item[(i)] $\Delta_{p}$ is a non-degenerate parabola if and only if $j^{2}f(0)\sim_{\mathcal{A}^{2}}(x,y^{2},xy)$;
\item[(ii)] $\Delta_{p}$ is a half-line if and only if $j^{2}f(0)\sim_{\mathcal{A}^{2}}(x,y^{2},0)$;
\item[(iii)] $\Delta_{p}$ is a line if and only if $j^{2}f(0)\sim_{\mathcal{A}^{2}}(x,xy,0)$;
\item[(iv)] $\Delta_{p}$ is a point if and only if $j^{2}f(0)\sim_{\mathcal{A}^{2}}(x,0,0)$.
\end{itemize}
\end{teo}
\begin{proof}
For completeness we give a sketch of the proof. Without loss of generality, assume that
$$j^{2}f(0)=\left(x,\frac{1}{2}(a_{20}x^{2}+2a_{11}xy+a_{02}y^{2}),\frac{1}{2}(b_{20}x^{2}+2b_{11}xy+b_{02}y^{2})\right).$$
Let $\{e_{1},e_{2},e_{3}\}$ be the standard basis of $\mathbb{R}^{3}$. Hence,
$T_{p}M=[e_{1}]$ and $N_{p}M=[e_{2},e_{3}]$ and the matrix of coefficients of the
second fundamental form is
$$\left(
    \begin{array}{ccc}
      a_{20} & a_{11} & a_{02} \\
      b_{20} & b_{11} & b_{02} \\
    \end{array}
  \right).
$$
According to \cite{mond}, the classification of $j^{2}f(0)$ follows from the analysis of the coefficients
$a_{ij},b_{ij}$:
\begin{itemize}
\item[(a)] $j^{2}f(0)\sim_{\mathcal{A}^{2}}(x,y^{2},xy)$ iff $a_{11}b_{02}-a_{02}b_{11}\neq0$;
\item[(b)] $j^{2}f(0)\sim_{\mathcal{A}^{2}}(x,y^{2},0)$ iff $a_{11}b_{02}-a_{02}b_{11}=0$ and $a_{02}^{2}+b_{02}^{2}>0$;
\item[(c)] $j^{2}f(0)\sim_{\mathcal{A}^{2}}(x,xy,0)$ iff $a_{02}=b_{02}=0$ and $a_{11}^{2}+b_{11}^{2}>0$;
\item[(d)] $j^{2}f(0)\sim_{\mathcal{A}^{2}}(x,0,0)$ iff $a_{02}=b_{02}=a_{11}=b_{11}=0$.
\end{itemize}
The result now follows from comparing those conditions with the parametrisation of $\Delta_{p}$
given by $\eta(y)=(0,a_{20}+2a_{11}y+a_{02}y^{2},b_{20}+2b_{11}y+b_{02}y^{2})$.
\end{proof}

Asymptotic and binormal directions for corank $1$ surfaces in $\mathbb{R}^{3}$ are defined
in terms of the second fundamental form. However, the next results show that they are inspired by
those of a regular surface in $\mathbb{R}^{4}$, where we have the curvature ellipse in the normal
plane.

We say that a non zero tangent direction $X\in T_{q}\tilde{M}$ is \emph{asymptotic}
if there is a non zero normal vector $\nu\in N_{p}M$ such that $II_{\nu}(X,Y)=0$,
for any $Y\in T_{q}\tilde{M}$. Moreover, in such case we say that $\nu$ is a
\emph{binormal direction}.

\begin{lem}[\cite{MartinsBallesteros}]\label{asymptotic2}
Let $\{\nu_{1},\nu_{2}\}$ be an orthonormal frame of $N_{p}M$. A tangent direction
$X=a\partial_{x}+b\partial_{y}\in T_{q}\tilde{M}$ is asymptotic if and only if
$$
\left|
  \begin{array}{ccc}
    b^{2} & -ab & a^{2} \\
    l_{\nu_{1}} & m_{\nu_{1}} & n_{\nu_{1}} \\
    l_{\nu_{2}} & m_{\nu_{2}} & n_{\nu_{2}} \\
  \end{array}
\right|=0,
$$
in which $l_{\nu_{i}},m_{\nu_{i}}$ and $n_{\nu{i}}$, $i=1,2$
are the coefficients of the second fundamental form.
\end{lem}

We can choose local coordinates for $\tilde{M}$ such that the curvature parabola is parametrised
in the normal plane by $\eta$, as given in (\ref{parabola}). The parameter value $y\in\mathbb{R}$
corresponds to a unit tangent direction $X =\partial_{x}+y\partial_{y}\in C_{q}$. We denote by
$y_{\infty}$ the parameter value corresponding to the null tangent direction $X =\partial_{y}$.
In the case that $\Delta_{p}$ degenerates to a line or a half-line, we define
$\eta(y_{\infty})=\eta'(y_{\infty})=\eta'(y)/|\eta'(y)|$, where $y>0$ is any value such that
$\eta'(y)\neq0$. In the case that $\Delta_{p}$ degenerates to a point $\nu$, then we define
$\eta(y_{\infty})=\nu$ and $\eta'(y_{\infty})=0$. In the case that $\Delta_{p}$ is a
non-degenerate parabola, $\eta(y_{\infty})$ and $\eta'(y_{\infty})$ are not defined.

\begin{lem}[\cite{MartinsBallesteros}]
A tangent direction in $T_{q}M$ given by a parameter value $y\in\mathbb{R}\cup[y_{\infty}]$ is
asymptotic if and only if $\eta(y)$ and $\eta'(y)$ are collinear (provided they are defined).
\end{lem}

The parameter $y\in\mathbb{R}\cup[y_{\infty}]$ corresponding to an asymptotic direction
$X\in T_{q}\tilde{M}$ is also called an asymptotic direction. It is possible to study
the asymptotic directions $y\in\mathbb{R}\cup[y_{\infty}]$ by studing each type of
curvature parabola.
\begin{itemize}
\item[(i)] If $\Delta_{p}$ is a non-degenerate parabola, we have $0,1$ or $2$ asymptotic
directions, acordind to the position of $p$: outside, on or outside the parabola,
repectively;
\item[(ii)] If $\Delta_{p}$ is a half-line, or we have two asymptotic directions,
$[y_{\nu},y_{\infty}]$, with $\eta(y_{\nu})$ being the vertex of $\Delta_{p}$ or
every $y\in\mathbb{R}\cup[y_{\infty}]$ is an asymptotic direction, according to
the line containing $\Delta_{p}$ does not pass through $p$ or it does, respectively;
\item[(iii)] If $\Delta_{p}$ is a line then either $y_{\infty}$ is the only
asymptotic direction or every $y\in\mathbb{R}\cup[y_{\infty}]$ is an asymptotic
direction, according to the line does not contain $p$ or it does, respectively.
\item[(iv)] If $\Delta_{p}$ is a point, every $y\in\mathbb{R}\cup[y_{\infty}]$
is an asymptotic direction.
\end{itemize}

In \cite{MartinsBallesteros} is shown that the height function $h_{v}:U\rightarrow\mathbb{R}$
given by $h_{v}(x,y)=\langle f(x,y),v\rangle$, where $f:U\subset\mathbb{R}^{2}\rightarrow\mathbb{R}$
is a local parametrisation of $M$ at $p$ and $v\in S^{2}$ is singular
at $p\in M$ if and only if $v\in N_{p}M$. Moreover, the singularity is degenerate
if and only if $v\in N_{p}M$ is a binormal direction (for $\Delta_{p}$ not being a point).
When $\Delta_{p}$ is a point, the singularity is degenerate for all $v\in N_{p}M$.

%
%

%

\section{Relation amongst the curvature loci}

In this section we relate the geometry of an immersed surface $N$ in $\mathbb R^4$ at a point $p$ to the geometry of the projection of the
surface to $\mathbb R^3$ at the singular point. In particular we relate the curvature ellipse of the surface to the curvature parabola of its singular projection.

We may assume that $p$ is the origin in $\mathbb R^4$ and that the parametrisation of $N$ is in Monge form so that $X(x,y)=(x,y,f_1(x,y),f_2(x,y))$ such that $f_1,f_2$ don't have constant or
linear part. Given a direction $\bf v\in S^3$, the orthogonal projection in $\mathbb R^4$ in the direction of $\bf v$ is given by $\pi_{\bf v}:\mathbb R^4 \rightarrow \mathbb R^3$ with $$\pi_{\bf v}(p)=p-\langle p, \bf v\rangle{\bf v}.$$ We want to study projection along tangent directions so without loss of generality we can assume by rotation in the tangent plane that ${\bf v}=(0,1)\in T_pN$, since rotation in the tangent plane leaves
invariant the curvature ellipse. Therefore, a parametrisation for $\pi_{\bf v}(N)$ is given by $f(x,y)=(x,f_1(x,y),f_2(x,y))$.

\begin{prop}\label{asymptotic}
Let $N\subset\mathbb{R}^{4}$ be a regular surface and $\pi_{\bf v}(N)$
its projection along the direction ${\bf v}=(0,1)\in T_pN$.
\begin{itemize}
\item[(i)] The number of asymptotic directions of $N$ at $p$ and of $\pi_{\bf v}(N)$
at $\pi_{\bf v}(p)$ is equal.
\item[(ii)] The number of binormal directions are equal on both surfaces.
Moreover, if $v\in N_{p}N$ is a binormal direction of $N$, it is also a binormal
direction of $\pi_{\bf v}(N)$ at $\pi_{\bf v}(p)$.
\end{itemize}
\end{prop}
\begin{proof}
For the first statement,
the proof follows from the facts that the coefficients of the second fundamental form
from both surfaces, $N$ and $\pi_{\bf v}(N)$, are equal at the corresponding points $p$ and
$\pi_{\bf v}(p)$ and the results that characterize asymptotic directions in both cases
depend only on those coefficients: Lemma \ref{asymptotic1} and Lemma \ref{asymptotic2}.

Following \cite{OsetSinhaTari}, the height function on the projected surface $\pi_{\bf v}(N)$ along the normal vector $\omega$ is given by $$\langle \pi_{\bf v}(x,y),\omega\rangle=\langle X(x,y)-\langle X(x,y),{\bf v}\rangle {\bf v},\omega\rangle=\langle X(x,y),\omega\rangle,$$ which is precisely the height function on $N$ along the vector $\omega$. Since binormal directions, both in $\mathbb R^4$ and $\mathbb R^3$, are given by the degenerate singularities of the height function, we get part ii).
\end{proof}

In the same way as for surfaces in $\mathbb R^4$, in \cite[Definition 2.1]{OsetSinhaTari} the second author and Tari classify singular points of corank 1 surfaces in $\mathbb R^3$ according to the $\mathcal G$-orbit of the pair $(Q_{1},Q_{2})$ where $Q_{1}(x,y)=j^{2}f_{1}(x,y)$ and $Q_{2}(x,y)=j^{2}f_{2}(x,y)$. They prove that a point $p\in N\subset \mathbb R^4$ is an elliptic/hyperbolic/parabolic/inflection point if and only if $\pi_{\bf v}(p)$ is of elliptic/hyperbolic/parabolic/inflection type (\cite[Theorem 3.3]{OsetSinhaTari}).

Proposition \ref{asymptotic} leads us to make the following

\begin{definition}\label{classr3}
Given a surface $M\subset\mathbb{R}^{3}$ with corank $1$ singularity at $p\in M$.
The point $p$ is called:
\begin{itemize}
\item[(i)] elliptic, if there are no asymptotic directions at $p$;
\item[(ii)] hyperbolic, if there are two asymptotic directions at $p$;
\item[(iii)] parabolic, if there is one asymptotic direction at $p$;
\item[(iv)] an inflection point, if there are infinite asymptotic directions at $p$.
\end{itemize}
\end{definition}

With this definition

\begin{teo}
Let $\bf v$ be a tangent direction at $p\in N\subset \mathbb R^4$. The point $p$ is an elliptic/hyperbolic/parabolic/inflection point if and only if the singular point $\pi_{\bf v}(p)\in \pi_{\bf v}(N)\subset \mathbb R^3$ is an elliptic/hyperbolic/parabolic/inflection point, respectively.
\end{teo}
\begin{proof}
Follows directly from i) in Proposition \ref{asymptotic} and Definition \ref{classr3}.
\end{proof}

\begin{coro}
Definition 2.1 in \cite{OsetSinhaTari} and Definition \ref{classr3} coincide.
\end{coro}

\begin{rem}
The cross-cap is a corank $1$ singularity whose geometry has been widely studied.
It is shown in \cite{West} that, by a suitable change of coordinates in the source and
an affine coordinate chance in the target, we can parametrisate a cross-cap in
the form $f(x,y)=(x,xy+p(y),y^{2}+ax^{2}+q(x,y))$, with $p\in\mathcal{M}^{4}$
and $q\in\mathcal{M}_{2}^{3}$. The cross-cap is called hyperbolic, elliptic or parabolic
if $a<0$, $a>0$ or $a=0$, respectively. In \cite{BallesterosTari} and \cite{OsetSinhaTari}, the authors show that
a cross-cap is hyperbolic, elliptic or parabolic if and only if its singular point
is elliptic, hyperbolic or parabolic respectively. This result is explained by our previous definition, since at a hyperbolic
cross-cap there are no asymptotic directions, an elliptic cross-cap has two and a parabolic cross-cap
has only one.
\end{rem}

We need a characterisation for when the curvature ellipse is degenerate.

\begin{prop}\label{degellipse}
The curvature ellipse $\Delta_e$ degenerates (to a segment or a point) if and only if $$(a_{20}b_{11}-b_{20}a_{11})+(a_{11}b_{02}-a_{02}b_{11})=0.$$
\end{prop}
\begin{proof}
The curvature ellipse is parametrised by
$\eta_e(\theta)=2(a_{20}\cos(\theta)^2+a_{11}\sin(\theta)\cos(\theta)+a_{02}\sin(\theta)^2,b_{20}\cos(\theta)^2+b_{11}\sin(\theta)\cos(\theta)+b_{02}\sin(\theta)^2).$
The ellipse is degenerate if and only if $\kappa=0$, where $\kappa$ is the curvature of $\eta_e$ seen as a plane curve. Now $\kappa=0$ if and only if
$\det(\eta_e,\eta_e')=0$ and a direct calculation shows that this is equivalent to $(a_{20}b_{11}-b_{20}a_{11})+(a_{11}b_{02}-a_{02}b_{11})=0.$
\end{proof}


\begin{figure}
\begin{center}
\includegraphics[width=1\linewidth]{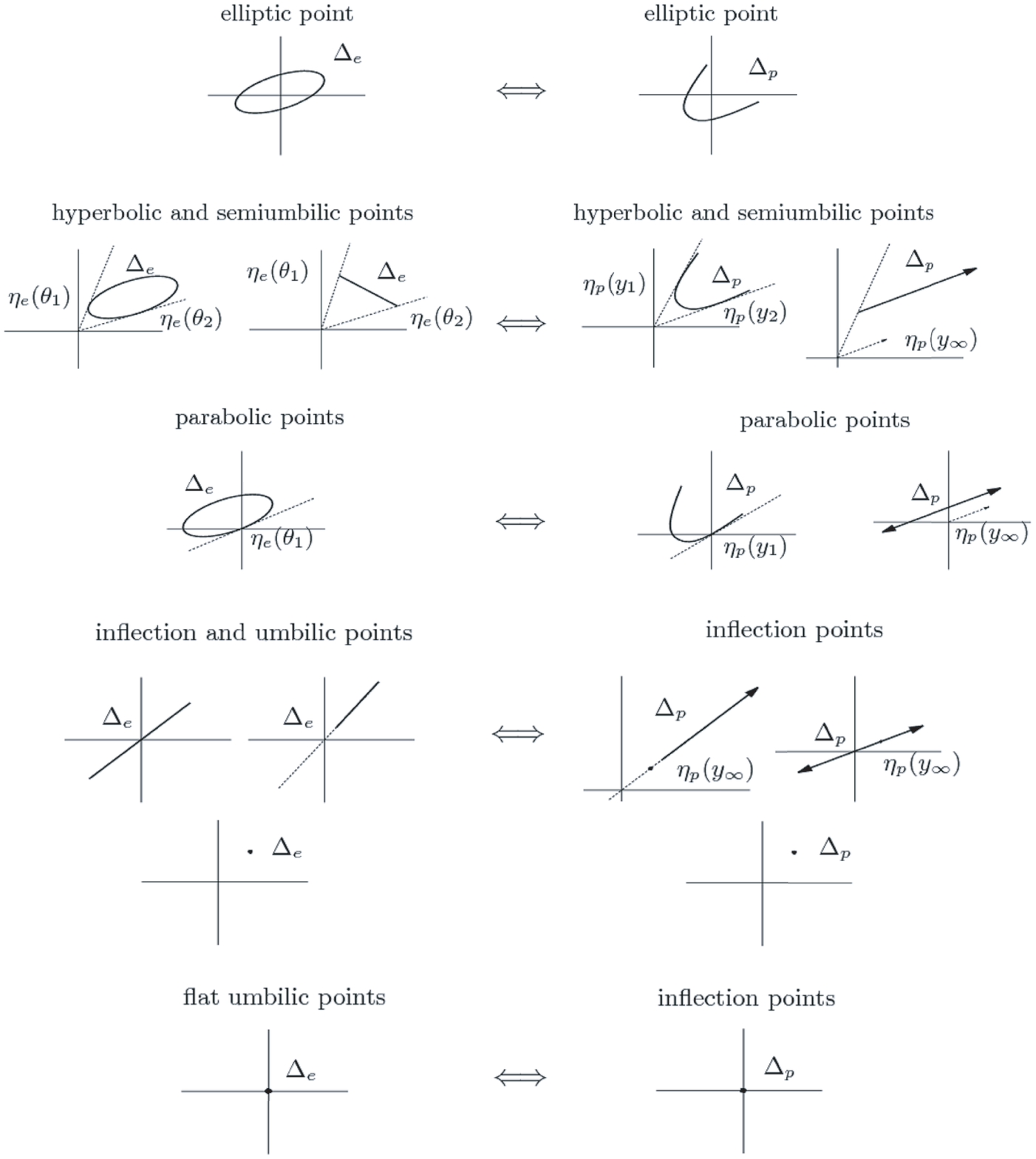}
\caption{Theorem \ref{main}.}
\label{ccc}
\end{center}
\end{figure}

A weaker version of Theorem \ref{main} can be deduced from Theorem \ref{projection}, i) in Theorem \ref{conditionsparabola} and i) in Proposition \ref{asymptotic}. However, we give a different proof here which gives further details into the geometry that will be used later on.

\begin{proof} (Proof of Theorem \ref{main})

First notice that by Theorem \ref{projection} if $\bf v$ is not an asymptotic direction, then $\pi_{\bf v}(N)$ has a cross-cap singularity at $\pi_{\bf v}(p)$
and so, by Theorem \ref{conditionsparabola}, the curvature parabola is non degenerate. If $\bf v$ is an asymptotic direction, then the singularity of the projection can be more degenerate than a
cross-cap and the curvature parabola is degenerate.

Notice too that by Proposition \ref{asymptotic} the number
of asymptotic directions at $p$ (0, 1, 2 or infinte) is the same as that of $\pi_{\bf v}(p)$.

i) By definition $\Delta_e$ is an ellipse with $p$ inside if $p$ is an elliptic point. In this case there are no asymptotic directions, so $\bf v$ is not
asymptotic and $\pi_{\bf v}(N)$ has a cross-cap singularity at $\pi_{\bf v}(p)$ and so the curvature parabola is non degenerate. Now, there are no asymptotic
directions at $\pi_{\bf v}(p)$ (it is an elliptic point) if and only if this point lies inside the curvature parabola $\Delta_p$.

ii) By definition $\Delta_e$ is an ellipse with $p$ lying outside it or a segment whose line does not contain $p$ if $p$ is a hyperbolic (possibly semiumbilic
but not inflection) point. In this case there are exactly 2 different asymptotic directions. By the previous section, $p$ is hyperbolic if and only if
$\Delta(p)=4(a_{20}b_{11}-b_{20}a_{11})(a_{11}b_{02}-a_{02}b_{11})-(a_{20}b_{02}-b_{20}a_{02})^2<0$. There are two possibilities, if
$a_{11}b_{02}-a_{02}b_{11}\neq 0$, by the proof of Theorem \ref{conditionsparabola} this condition holds if and only if $\Delta_p$ is a non-degenerate parabola.
Since there must be two asymptotic directions, the point $\pi_{\bf v}(p)$ must lie outside the parabola. This happens when $\bf v$ is not an asymptotic
direction.

On the other hand, if $a_{11}b_{02}-a_{02}b_{11}=0$, then $a_{20}b_{02}-b_{20}a_{02}\neq 0$ in order for $\Delta(p)$ to be negative, and so
$a_{02}^2+b_{02}^2>0$. This condition is equivalent by Theorem \ref{conditionsparabola} to $\Delta_p$ being a half line, and since there must be 2 asymptotic
directions, the line which contains this half-line does not contain the point $\pi_{\bf v}(p)$. In this case $\bf v$ is an asymptotic direction.

Reciprocally, if $\Delta_p$ is a parabola with $\pi_{\bf v}(p)$ lying outside it or a half-line whose line does not contain $\pi_{\bf v}(p)$, then there are
exactly 2 asymptotic directions and therefore $p$ must be a hyperbolic point (possibly semiumbilic but not inflection).

iii) By definition $\Delta_e$ is an ellipse with $p$ lying on it if $p$ is a parabolic (not inflection) point, i.e. there is only 1 asymptotic direction. These
points are characterized by $\Delta(p)=4(a_{20}b_{11}-b_{20}a_{11})(a_{11}b_{02}-a_{02}b_{11})-(a_{20}b_{02}-b_{20}a_{02})^2=0$. Again we study two cases, if
$a_{11}b_{02}-a_{02}b_{11}\neq 0$ then $\Delta_p$ is a non-degenerate parabola and since there is only 1 asymptotic direction the point $\pi_{\bf v}(p)$ must lie
on the parabola. Here $\bf v$ is not an asymptotic direction.

On the other hand, if $a_{11}b_{02}-a_{02}b_{11}=0$ then $a_{20}b_{11}-b_{20}a_{11}\neq 0$. This is because by Proposition \ref{degellipse} $\Delta_e$ is
non-degenerate if and only if $(a_{20}b_{11}-b_{20}a_{11})+(a_{11}b_{02}-a_{02}b_{11})\neq 0.$ Therefore $a_{11}^2+b_{11}^2>0$. Since $\Delta(p)=0$ then
$a_{20}b_{02}-b_{20}a_{02}=0$. This implies that $a_{02}=b_{02}=0$. By Theorem \ref{conditionsparabola}, $a_{11}^2+b_{11}^2>0$ and $a_{02}=b_{02}=0$ if and only
if $\Delta_p$ is a line, and since there is only 1 asymptotic direction, this line does not contain $\pi_{\bf v}(p)$. Here $\bf v$ is an asymptotic direction.

Reciprocally, if $\Delta_p$ is a parabola with $\pi_{\bf v}(p)$ lying on it or a line which does not contain $\pi_{\bf v}(p)$ then there is exactly 1 asymptotic
direction and therefore $p$ must be a parabolic (not inflection) point.

iv) By definition $\Delta_e$ is a segment whose line contains $p$ or a point different from $p$ if $p$ is an inflection (possibly umbilic but not flat umbilic)
point. In this case $\Delta(p)=0$ and since $\Delta_e$ is degenerate $(a_{20}b_{11}-b_{20}a_{11})+(a_{11}b_{02}-a_{02}b_{11})=0.$ Therefore
$a_{20}b_{11}-b_{20}a_{11}=0$, $a_{11}b_{02}-a_{02}b_{11}=0$ and $a_{20}b_{02}-b_{20}a_{02}=0$. These 3 conditions together imply that all tangent directions are
asymptotic. If $a_{02}^2+b_{02}^2>0$ $\Delta_p$ is a half-line and since all directions are asymptotic $\pi_{\bf v}(p)$ is contained in the line which contains
this half-line. If $a_{02}=b_{02}=0$ and $a_{11}^2+b_{11}^2>0$ $\Delta_p$ is a line and $p$ will be contained in it. If $a_{02}=b_{02}=a_{11}=b_{11}=0$ then
$\Delta_p$ is a point. In this case $a_{20}^2+b_{20}^2\neq 0$ because otherwise $\Delta_e=\{p\}$, so $\Delta_p$ is a point different from $\pi_{\bf v}(p)$.

v) If $\Delta_e=\{p\}$, $p$ is a flat umbilic by definition and $a_{02}=b_{02}=a_{11}=b_{11}=a_{20}=b_{20}=0$, therefore $\Delta_p=\{\pi_{\bf v}(p)\}$.
\end{proof}

\begin{ex}
\begin{enumerate}
\item[i)] Consider the surface in $\R^4$ parametrised by $(x,y,x^2+xy+y^2,x^2+2xy+y^2)$. This surface has a semiumbilic point at the origin and the
    curvature ellipse is parametrised by
    $\eta_e(\theta)=2(\cos(\theta)^2+\sin(\theta)\cos(\theta)+\cos{\theta}^2,\cos(\theta)^2+2\sin(\theta)\cos(\theta)+\cos{\theta}^2)$, which is a segment
    uniting the points $(\frac{1}{2},0)$ and $(\frac{3}{2},2)$ in the normal plane. Therefore $(0,1)$ is not an asymptotic direction. Projecting along
    $(0,1)$ yields $(x,x^2+xy+y^2,x^2+2xy+y^2)$. The curvature parabola at the origin is parametrised by $\eta_p(y)=(0,y^2+y+1,y^2+2y+1)$, which is a
    non-degenerate parabola.
\item[ii)] Consider a surface in $\R^4$ whose 2-jet is parametrised by $(x,y,x^2,x^2)$ which has an inflection point at the origin. The curvature ellipse is
    the segment which goes from $(0,0)$ to $(2,2)$. Here all tangent directions are asymptotic. If we project along $(0,1)$ we get the surface in $\R^3$
    whose 2-jet is parametrised by $(x,x^2,x^2)$. The curvature parabola in this case is the point $\{(2,2)\}$. If we project along $(1,0)$, we get
    $(y,x^2,x^2)$, whose curvature parabola is the half-line given by $\eta_p(y)=(0,y^2,y^2)$.
\item[iii)] Consider a surface in $\R^4$ whose 2-jet is parametrised by $(x,y,x^2+y^2,x^2+y^2)$ which has an umbilic (non flat) point at the origin. The
    curvature ellipse is the point $\{(2,2)\}$ and all tangent directions are asymptotic. If we project along $(0,1)$ we get a surface parametrised by
    $(x,x^2+y^2,x^2+y^2)$ whose curvature parabola is the half-line given by $\eta_p(y)=(0,y^2+1,y^2+1)$.
\end{enumerate}
\end{ex}

Next we show the relation between the parametrisations of the curvature ellipse and the curvature parabola.

\begin{prop}
Consider an immersed surface in $\R^4$ given in Monge form and
consider the curvature ellipse parametrised by $\eta_e(\theta)=$
$$2(a_{20}\cos(\theta)^2+a_{11}\sin(\theta)\cos(\theta)+a_{02}\sin(\theta)^2,b_{20}\cos(\theta)^2+b_{11}\sin(\theta)\cos(\theta)+b_{02}\sin(\theta)^2).$$
The curvature parabola of the projection along the tangent direction
$(0,1)$ is parametrised by
$$\eta_p(y)=(0,2a_{20}+2a_{11}y+2a_{02}y^2,2b_{20}+2b_{11}y+2b_{02}y^2).$$
\end{prop}
\begin{proof}
The curvature ellipse is parametrised by the unit vectors in $T_pN$,
i.e. $\eta_e:S^1\subset T_pN\rightarrow N_pN$. On the other hand the
curvature parabola is also parametrised by unit vectors but in the
tangent space of $\tilde M$ we have a pseudo-metric induced by the
first fundamental form, not a metric. Namely, if the projection is
parametrised by $f(x,y)=(x,f_1(x,y),f_2(x,y))$, then $E=1$, $F=G=0$,
and so $I(X,X)=a^2$, where $X=a\partial_x+b\partial_y\in T_{q}\tilde
M$. So unit vectors are vectors in $C_q=\{(\pm 1,y): y\in\R\}$.
Direct computation in each case shows that the curvature parabola is
obtained from the curvature ellipse by dividing each component by
$\cos(\theta)^2$ and the change $\tan(\theta)=y$. In other words, we
change $S^1\subset T_pN$ to homogeneous coordinates of the
projective line when $\cos(\theta)\neq 0$, and $\cos(\theta)=0$
corresponds to the null tangent direction $X=\partial_y$.
\end{proof}

When projecting along an asymptotic direction the projected surface generically has one of the singularities in Table \ref{tab:Mondcodimle3}.
Notice that $S_k, B_k$, and $C_3$ have 2-jet equivalent to $(x,y^2,0)$ and hence the associated curvature parabola is a half-line. Therefore, these points can only be a hyperbolic point or an inflection point depending on wether there are 2 or infinite asymptotic directions respectively. According to the proof of Theorem \ref{main}, the first case occurs if $a_{20}b_{02}-b_{20}a_{02}\neq 0$ and the second case occurs if $a_{20}b_{02}-b_{20}a_{02}=0$ and $a_{02}^2+b_{02}^2>0$.

On the other hand, $H_k$ and $P_3(c)$ have 2-jet equivalent to
$(x,xy,0)$, and so the curvature parabola is a line, so these points
can only be a parabolic or an inflection point depending on wether
there is 1 or infinite asymptotic directions. According to the proof
of Theorem \ref{main}, this happens when
$a_{20}b_{11}-b_{20}a_{11}\neq 0$, $a_{02}=b_{02}=0$ and
$a_{11}^2+b_{11}^2>0$ for the parabolic case and when
$a_{20}b_{11}-b_{20}a_{11}=0$, $a_{02}=b_{02}=0$ and
$a_{11}^2+b_{11}^2>0$ for the inflection case.

Therefore, the proof of Theorem \ref{main} allows us to recover Theorem 2.5 in \cite{OsetSinhaTari} which states precisely the above discussion.

\end{document}